\documentclass[12pt, a4paper]{article}

\usepackage{amsmath,amsfonts,amsthm}
\usepackage[T1]{fontenc} 

\newtheorem{theorem}{Theorem}
\newtheorem{proposition}{Proposition}

\newtheorem{lemma}{Lemma}
\newtheorem{remark}{Remark}

\newcommand{\argmax}{\mathop{\mathrm{argmax}}}

\newcommand{\bbE}{\mathbb{E}}
\newcommand{\bbF}{\mathbb{F}}

\newcommand{\bbP}{\mathbb{P}}

\newcommand{\bbR}{\mathbb{R}}

\begin{document}

\title{Maximum likelihood estimators for the extreme value index based on the block maxima method}
\author{ Cl\'ement Dombry\footnote{Universit\'e de Poitiers, Laboratoire de Mathématiques et Applications, UMR CNRS 7348, T\'el\'eport 2, BP 30179, F-86962 Futuroscope-Chasseneuil cedex, France.  
Email: clement.dombry@math.univ-poitiers.fr} } 
\maketitle
\date{}
\abstract{The maximum likelihood method offers a standard way to estimate the three parameters of a generalized extreme value (GEV) distribution. Combined with the block maxima method, it is often used in practice to assess the extreme value index and normalization constants of a distribution satisfying a first order extreme value condition, assuming implicitely that the block maxima are exactly GEV distributed. This is unsatisfactory since the GEV distribution is a good approximation of the block maxima distribution only for blocks of large size. The purpose of this paper is to provide a theoretical basis for this methodology. Under a first order extreme value condition only, we prove the existence and consistency of the maximum likelihood estimators for the extreme value index and normalization constants  within the framework of the block maxima method.}

\vspace{0.5cm}


\section{Introduction and results}
Estimation of the extreme value index is a central problem in extreme value theory. A variety of estimators are available in the literature, for example among others, the Hill estimator \cite{H75}, the Pickand's estimator \cite{P75}, the probability weighted moment estimator introduced by Hosking {\it et al.} \cite{HWW85} or the moment estimator suggested by Dekkers {\it et al.} \cite{DEdH89}. The monographs by Embrechts {\it et al.} \cite{EKM97}, Beirlant {\it et al.} \cite{BGT04} or de Haan and Ferreira \cite{dHF06} provide good reviews on this estimation problem.

In this paper, we are interested on estimators based on the maximum likelihood method. Two different types of maximum likelihood estimators (MLEs) have been introduced, based on the peak over threshold method and block maxima method respectively. The peak over threshold method relies on the fact that, under the extreme value condition, exceedances over high threshold converge to a generalized Pareto distribution (GPD) (see Balkema and de Haan \cite{BdH74}). A MLE within the GPD model has been proposed by Smith \cite{S87}. Its theoretical properties under the extreme value condition are quite difficult to analyze due to the absence of an explicit expression of the likelihood equations: existence and consistency have been proven by Zhou \cite{Z09}, asymptotic normality  by Drees {\it et al.} \cite{DFdH04}. The block maxima method relies on the approximation of the maxima distribution by a generalized extreme value (GEV) distribution. Computational issues for ML estimation within the GEV model have been considered by Prescott and Walden \cite{PW80,PW83}, Hosking \cite{H85} and Macleod \cite{M89}. Since the support of the GEV distribution depends on the unknown extreme value index $\gamma$, the usual regularity conditions ensuring good asymptotic properties  are not satisfied. This problem is studied by Smith \cite{S85}: asymptotic normality is proven for $\gamma>-1/2$ and consistency for $\gamma>-1$.

It should be stressed that the block maxima method is based on the assumption that the observations come from a distribution satisfying the extreme value condition so that the maximum of a large numbers of observations follows approximatively a generalized extreme value (GEV) distribution. On the contrary, the properties of the maximum likelihood relies implicitely on the assumption that the  block maxima have {\it exactly} a GEV distribution. In many situations, this strong assumption is unsatisfactory and we shall only suppose that the underlying distribution is in the domain of attraction of an extreme value distribution. This is the purpose of the present paper to justify the maximum likelihood method for the block maxima method under an extreme value condition only.

We first recall some basic notions of univariate extreme value theory. The extreme value distribution distribution with index $\gamma$ is noted $G_\gamma$ and has distribution function
\[
F_{\gamma}(x)=\exp(-(1+\gamma x)^{-1/\gamma}),\quad 1+\gamma x>0.
\]
We say that a distribution function $F$ satisfies the extreme value condition with index $\gamma$, or equivalently that $F$ belongs to the domain of attraction of $G_\gamma$ if there exist constants $a_m>0$ and $b_m$ such that
\begin{equation}\label{eq:DG}
\lim_{m\to +\infty} F^m(a_mx+b_m)=F_\gamma(x),\quad x\in\bbR.
\end{equation}
That is commonly denoted $F\in D(G_\gamma)$. The necessary and sufficient conditions for $F\in D(G_\gamma)$ can be presented in different ways, see e.g. de Haan \cite{dH84} or de Haan and Ferreira \cite[chapter 1]{dHF06}. We remind the following simple criterion and choice of normalization constants. 
\begin{theorem}\label{theo:DA}
Let $U=\Big(\frac{1}{1-F}\Big)^\leftarrow$ be the left continuous inverse function of $1/(1-F)$. Then $F\in D(G_\gamma)$ if and only if there exists a function $a(t)>0$ such that
\[
\lim_{t\to+\infty} \frac{U(tx)-U(t)}{a(t)}=\frac{x^\gamma-1}{\gamma},\quad \mbox{for all}\ x>0.
\]
Then, a possible choice for the function  $a(t)$ is given by
\[
a(t)=\left\{\begin{array}{ll} \gamma U(t),& \gamma>0,\\ -\gamma(U(\infty)-U(t)), & \gamma <0,\\ U(t)-t^{-1}\int_0^{t}U(s)ds, &\gamma=0,\end{array}\right.
\]
and a possible choice for the normalization constants in \eqref{eq:DG} is 
\[
a_m=a(m) \quad\mbox{and}\quad b_m=U(m).
\]
\end{theorem}
\noindent
In the sequel, we will always use the normalization constants $(a_m)$ and $(b_m)$ given in Theorem \ref{theo:DA}. Note that they are unique up to asymptotic equivalence in the following sense: if $(a'_m)$ and $(b'_m)$ are such that $F^m(a_m'x+b_m')\to F_\gamma(x)$ for all $x\in\bbR$, then
\begin{equation}\label{eq:normseq}
\lim_{m\to +\infty} \frac{a_m'}{a_m}=1\quad \mbox{and} \lim_{m\to +\infty} \frac{b_m'-b_m}{a_m}=0.
\end{equation}

The log-likelihood of the extreme value distribution $G_\gamma$ is given by
\[
 \ell_{\gamma}(x)=-(1+1/\gamma)\log(1+\gamma x)-(1+\gamma x)^{-1/\gamma},
\]
if $1+\gamma x>0$ and $-\infty$ otherwise. For $\gamma=0$, the formula is interpreted as  $\ell_{0}(x)=- x-\exp(-x)$.
The three parameter extreme value distribution with shape  $\gamma$, location $\mu$ and scale $\sigma>0$ has distribution function $x\mapsto F_\gamma(\sigma x+\mu)$. The corresponding log-likelihood is
\[
 \ell_{(\gamma,\mu,\sigma)}(x)=\ell_\gamma\Big(\frac{x-\mu}{\sigma}\Big)-\log\sigma.
\] 

The set-up of the block maxima method is the following. We consider  independent and identically distributed (i.i.d.)  random variables $(X_i)_{i\geq 1}$ with common distribution function $F\in D(G_{\gamma_0})$ and corresponding normalization sequences $(a_m)$ and $(b_m)$ as in Theorem \ref{theo:DA}. We divide the sequence $(X_i)_{i\geq 1}$  into blocks of length $m\geq 1$ and define the $k$-th block maximum by
\[
 M_{k,m}=\max(X_{(k-1)m+1},\ldots,X_{km}),\quad k\geq 1.
\]
For fixed $m\geq 1$, the variables $(M_{k,m})_{k\geq 1}$ are i.i.d. with distribution function $F^m$ and  
\begin{equation}\label{eq:convBM}
\frac{M_{k,m}-b_m}{a_m}\Longrightarrow G_{\gamma_0}\quad \mbox{as}\ m\to +\infty.
\end{equation}
Equation \eqref{eq:convBM} suggests that the distribution of  $M_{k,m}$ is approximately a GEV distribution with parameters $(\gamma_0,b_m,a_m)$ and this is standard to estimate these parameters by the maximum likelihood method. The log-likelihood of the $n$-sample  $(M_{1,m},\ldots,M_{n,m})$ is 
\[
L_n(\gamma,\sigma,\mu)=\frac{1}{n}\sum_{k=1}^n  \ell_{(\gamma,\mu,\sigma)}(M_{k,m}).
\]
In general,  $L_n$ has no global maximum, leading us to the following weak notion: we say that $(\widehat\gamma_n,\widehat \mu_n,\widehat\sigma_n)$ is a MLE  if  $L_n$ has a \emph{local} maximum at $(\widehat\gamma_n,\widehat \mu_n,\widehat\sigma_n)$. Clearly, a MLE solves the  likelihood equations
\begin{equation}\label{eq:likeq}
\nabla L_n=0 \quad \mbox{with}\quad \nabla L_n=\Big(\frac{\partial L_n}{\partial \gamma},\frac{\partial L_n}{\partial \mu},\frac{\partial L_n}{\partial \sigma}\Big).
\end{equation}
Conversely, any solution of the likelihood equations with a definite negative Hessian matrix is a MLE.

For the purpose of asymptotic, we let the length of the blocks $m=m(n)$ depend on the sample size $n$. Our main result is the following theorem, stating the existence of consistent MLEs. 
\begin{theorem}\label{theo1}
Suppose $F\in D(G_{\gamma_0})$ with  $\gamma_0>-1$  and assume that 
\begin{equation}\label{eq:growth}
\lim_{n\to +\infty} \frac{m(n)}{\log n}=+\infty.
\end{equation}
Then there exists  a sequence of estimators $(\widehat\gamma_n,\widehat \mu_n,\widehat\sigma_n)$ and a random integer $N\geq 1$ such that
\begin{equation}\label{eq:theo1.1}
\bbP[(\widehat\gamma_n,\widehat \mu_n,\widehat\sigma_n)\ \mbox{is a MLE for all}\ n\geq N ]= 1 
\end{equation}
and
\begin{equation}\label{eq:theo1.2}
\widehat \gamma_n\stackrel{a.s.}\longrightarrow \gamma_0,\quad \frac{\widehat\mu_n-b_m}{a_m}\stackrel{a.s.}\longrightarrow 0\quad \mbox{and}\quad \frac{\widehat\sigma_n}{a_m}\stackrel{a.s.}\longrightarrow 1\quad \mbox{as} \ n\to+\infty.
\end{equation}
\end{theorem}
The condition $\gamma_0>-1$ is natural and agrees with Smith \cite{S85}: it is easy to see that the likelihood equation \eqref{eq:likeq} has no solution with $\gamma\leq 1$ so that no consistent MLE exists when $\gamma_0<-1$ (see Remark \ref{rk3} below). Condition \eqref{eq:growth} states that the block length $m(n)$ grows faster than logarithmically in the sample size $n$, which is not very restrictive. Let us mention a few further remarks on this condition.
\begin{remark}\label{rk1}{\rm
A control of the block size is needed, as the following simple example shows. 
Consider a distribution $F\in D(G_{\gamma_0})$ with $\gamma_0>0$ and such that the left endpoint of $F$ is equal to $-\infty$. Then for each $m\geq 1$, the distribution of the $m$-th block maxima $M_{k,m}$ has left endpoint equal to $-\infty$ and there exist a sequence $m(n)$ (growing slowly to $+\infty$) such that
\begin{equation}\label{eq:diverge}
\lim_{n\to +\infty}\min_{1\leq k\leq n} \frac{M_{k,m}-b_m}{a_m} =-\infty \quad \mbox{almost surely.}
\end{equation}
The  log-likelihood  $L_n(\gamma,\sigma,\mu)$ is finite if and only if
\[
\min_{1\leq k\leq n}\Big( 1+\gamma \frac{M_{k,m}-\mu}{\sigma}\Big) >0,
\]
so that any MLE $(\widehat\gamma_n,\widehat \mu_n,\widehat\sigma_n)$ must satisfy
\[
\min_{1\leq k\leq n}\Big( 1+\widehat\gamma_n \frac{M_{k,m}-\widehat\mu_n}{\widehat\sigma_n}\Big) >0.
\]
Using this observations, one shows easily that Equation \eqref{eq:diverge} is an obstruction for the consistency \eqref{eq:theo1.2} of the MLE. Of course, this phenomenon can not happen under condition \eqref{eq:growth}.}
\end{remark}
\begin{remark}\label{rk2}{\rm
It shall be stressed that condition \eqref{eq:growth} appears only in the proof of Lemma~\ref{lem:crucial} below. One can prove that under stronger assumptions on the distribution $F\in D(G_{\gamma_0})$,  condition \eqref{eq:growth}. This is for example the case if $F$ is a Pareto distribution function: one checks easily that the proof of Lemma~\ref{lem:crucial} goes through  under the weaker condition  $\lim_{n\to +\infty} m(n)=+\infty$. Hence Theorem~\ref{theo1} holds under this weaker condition in the Pareto case. In order to avoid technical conditions that are hard to check in practice when $F$ is unknown, we do not develop this direction  any further.}
\end{remark}

The structure of the paper is the following. We gather in Section 2 some preliminaries on properties of the GEV log-likelihood and of the empirical distribution associated to normalized block maxima. Section 3 is devoted to the proof of Theorem \ref{theo1}, which relies on an adaptation of Wald's method for proving the consistency of $M$-estimators.
Some technical computations (proof of Lemma \ref{lem:crucial}) involving regular variations theory are postponed to an Appendix.

\section{Preliminaries}
\subsection{Properties of the GEV log-likelihood}
We gather in the following proposition some basic properties of the GEV log-likelihood. We note $x_\gamma^-$ and $x_\gamma^+$  the left and right end point of the domain $\ell_\gamma$, i.e.
\[
(x_\gamma^-,x_\gamma^+)=\{x\in\bbR;\ 1+\gamma x>0\}.
\]
Clearly,  it is equal to  $(-\infty,-1/\gamma)$, $\bbR$ and $(-1/\gamma,+\infty)$  when   $\gamma<0$,  $\gamma=0$ and  $\gamma>0$ respectively. 
\begin{proposition}\label{prop:GEV}
The function $\ell_\gamma$ is smooth on its domain. 
\begin{enumerate}
\item If $\gamma\leq -1$,  $\ell_\gamma$ is stricly increasing on its domain and
\[
 \lim_{x\to x_\gamma^-}\ell_\gamma(x)=-\infty\quad \quad \quad \lim_{x\to x_\gamma^+}\ell_\gamma(x)=\left\{\begin{array}{cc}+\infty & \mbox{if } \gamma<-1\\0 & \mbox{if } \gamma=-1 \end{array}\right.. 
\]
\item If $\gamma >-1$,  $\ell_\gamma$ is  increasing on $(x_\gamma^-,x_\gamma^\ast]$ and  decreasing  on $[x_\gamma^\ast,x_\gamma^+)$, where
\[
x_\gamma^\ast= \frac{ (1+\gamma)^{-\gamma}-1}{\gamma}.
\]
Furthermore
\[
 \lim_{x\to x_\gamma^-}\ell_\gamma(x)=\lim_{x\to x_\gamma^+}\ell_\gamma(x)=-\infty
\]
and $\ell_\gamma$ reaches its maximum $ \ell_\gamma(x_\gamma^\ast)=(1+\gamma)(\log(1+\gamma)-1)$ uniquely.
\end{enumerate}
\end{proposition}

\begin{remark}\label{rk3}{\rm
According to Proposition \ref{prop:GEV}, the log-likelihood $\ell_{\gamma}$ has no local maximum in the case $\gamma\leq -1$. This entails that the log-likelihood Equation \eqref{eq:likeq} has no local maximum  in $(-\infty,-1]\times\bbR\times (0,+\infty)$ and that any MLE $(\widehat\gamma_n,\widehat \mu_n,\widehat\sigma_n)$ satisfies $\widehat\gamma_n>-1$. Hence, no consistent MLE does exist if $\gamma_0<-1$. The limit case $\gamma_0=-1$ is more difficult to analyze and is disregarded in this paper.}
\end{remark}

\subsection{Normalized block maxima}
In view of Equation \eqref{eq:convBM}, we define the normalized block maxima
\[
\widetilde M_{k,m}=\frac{M_{k,m}-b_m}{a_m},\quad k\geq 1, 
\]
and the corresponding likelihood
\[
\widetilde L_n(\gamma,\sigma,\mu)=\frac{1}{n}\sum_{k=1}^n  \ell_{(\gamma,\mu,\sigma)}(\widetilde M_{k,m(n)}).
\]
It should be stressed that the normalization sequences $(a_m)$ and $(b_m)$ are unknown so that the normalized block maxima $\widetilde M_{k,m}$ and the likelihood $\widetilde L_n$ cannot be computed  from the data only. However, they will be useful in our theoretical analysis since they have  good asymptotic properties. The following simple observation will be useful. 
\begin{lemma}\label{lem:basic}
$(\widehat \gamma_n,\widehat \mu_n,\widehat\sigma_n)$ is a MLE if and only if $\widetilde L_n$ has a local maximum at $(\widehat \gamma_n, (\widehat\mu_n-b_m)/a_m,\widehat\sigma_n/a_m)$.
\end{lemma}
\begin{proof}
The GEV likelihood satisfies the scaling property
\[
 \ell_{\gamma,\mu,\sigma}((x-b)/a)=\ell_{(\gamma,a\mu+b,a\sigma)}(x)+\log a
\]
so that
\[
 L_n(\gamma,\mu,\sigma)= \widetilde L_n\Big(\gamma,\frac{\mu-b_m}{a_m},\frac{\sigma}{a_m}\Big)-\log a_m.
\]
Hence the local maximizers of $L_n$ and  $\widetilde L_n$ are in direct correspondence and the lemma follows.
\end{proof}

\subsection{Empirical distributions}
The likelihood function $\tilde L_n$ can be seen as a functional of the empirical distribution  defined by
\[
\bbP_n=\frac{1}{n}\sum_{k=1}^n \delta_{\widetilde M_{k,m}},
\]
where $\delta_x$ denotes the Dirac mass at point $x\in\bbR$. For any measurable $f:\bbR\to [-\infty,+\infty)$, we note $\bbP_n[f]$ the integral with respect to $\bbP_n$, i.e.
\[
\bbP_n[f]=\frac{1}{n}\sum_{k=1}^n f(\widetilde M_{k,m}).
\]
With these notations, it holds
\[
 \widetilde L_n(\gamma,\mu,\sigma)=\bbP_n[\ell_{(\gamma,\mu,\sigma)}].
\]
The empirical process is defined by 
\[
\bbF_n(t)=\bbP_n((-\infty,t])=\frac{1}{n}\sum_{k=1}^n 1_{\{\widetilde M_{k,m}\leq t\}},\quad t\in\bbR.
\]
In the case of an i.i.d. sequence, the Glivenko-Cantelli Theorem  states that the empirical process converges almost surely uniformly to the sample distribution function. According to the general theory of empirical processes (see {\it e.g.} Shorack and Wellner \cite{SW86} Theorem 1, p106), this result can be extended to triangular arrays of i.i.d. random variables. Equation \eqref{eq:convBM} entails the following result.
\begin{lemma}\label{prop1}
Suppose $F\in D(G_{\gamma_0})$ and $\lim_{n\to+\infty}m(n)=+\infty$. Then,
\[
\sup_{t\in\bbR}|\bbF_n(t) -F_{\gamma_0}(t)| \stackrel{a.s.}\longrightarrow 0 \quad \mbox{as}\ n\to +\infty.
\]
\end{lemma}
\noindent
This entails the almost surely weak convergence $\bbP_n\Rightarrow G_{\gamma_0}$, whence 
\[
\bbP_n[f]\stackrel{a.s}\longrightarrow G_{\gamma_0}[f] \quad \mbox{as}\ n\to +\infty
\]
for all bounded and continuous function $f:\bbR\to\bbR$. The following lemma dealing with more general functions will be useful.
\begin{lemma}\label{lem:Fatou}Suppose $F\in D(G_{\gamma_0})$ and $\lim_{n\to+\infty}m(n)=+\infty$. Then, for all upper semi-continuous  function $f:\bbR\to[-\infty,+\infty)$  bounded from above,
\[
\limsup_{n\to+\infty} \bbP_n[f] \leq G_{\gamma_0}[f] \quad \mbox{a.s.}.
\] 
\end{lemma}
\begin{proof}[Proof of Lemma \ref{lem:Fatou}]
Let $M$ be an upper bound for $f$. The function $\tilde f=M-f$ is non-negative and lower semicontinuous. Clearly,
\[
 \bbP_n[f]=M-\bbP_n[\tilde f] \quad \mbox{and}\quad G_{\gamma_0}[f]=M-G_{\gamma_0}[\tilde f],
\]
whence it is enough to prove that
\[
\liminf_{n\to+\infty} \bbP_n[\tilde f] \geq G_{\gamma_0}[\tilde f] \quad \mbox{a.s.}.
\] 
To see this, we use the relation
\[
 \bbP_n[\tilde f]=\int_0^1 \tilde f(\bbF_n^\leftarrow(u))du.
\]
where $\bbF_n^\leftarrow$ is the left-continuous  inverse function
\[
 \bbF_n^\leftarrow=\inf\{x\in \bbR;\ \bbF_n(x)\geq u\},\quad u\in (0,1).
\]
Lemma \ref{prop1} together with the continuity of the distribution function $F_{\gamma_0}$ entail that almost surely, 
$ \bbF_n^\leftarrow(u)\to F_{\gamma_0}^\leftarrow(u)$ for all $u\in(0,1)$ as $n\to +\infty$. Using the fact that $\tilde f$ is lower semi-continuous, we obtain
\[
 \liminf_{n\to +\infty} \tilde f(\bbF_n^\leftarrow(u))\geq \tilde f (F_{\gamma_0}^\leftarrow(u)) \quad u\in(0,1).
\]
On the other hand, according to Fatou's lemma,
\[
 \liminf_{n\to+\infty}\int_0^1 \tilde f(\bbF_n^\leftarrow(u))du \geq \int_0^1 \liminf_{n\to+\infty}\tilde f(\bbF_n^\leftarrow(u))du.
\]
Combining the two inequalities, we obtain
\[
 \liminf_{n\to +\infty} \tilde f(\bbF_n^\leftarrow(u))\geq \int_0^1\tilde f (F_{\gamma_0}^\leftarrow(u))du,
\]
whence 
\[
\liminf_{n\to+\infty} \bbP_n[\tilde f] \geq G_{\gamma_0}[\tilde f] \quad \mbox{a.s.}.
\] 
\end{proof}
The next lemma plays a crucial role in our proof of Theorem~\ref{theo1}. Its proof is quite technical and is postponed to an appendix. 
\begin{lemma}\label{lem:crucial}
Suppose $F\in  D(G_{\gamma_0})$ with  $\gamma_0>-1$  and assume  condition \eqref{eq:growth} is satisfied. Then,
\begin{equation}\label{eq:crucial}
\lim_{n\to +\infty}\bbP_n[\ell_{\gamma_0}]=G_{\gamma_0}[\ell_{\gamma_0}]\quad \mbox{a.s.}.
\end{equation}
\end{lemma}
\noindent
It shall be stressed that Lemma \ref{lem:crucial} is the only part in the proof of Theorem \ref{theo1} where  condition \eqref{eq:growth} is needed (see Remark \ref{rk2}).

\section{Proof of Theorem \ref{theo1}}
We introduce the short notation $\Theta=(-1,+\infty)\times\bbR\times (0,+\infty)$. A generic point of $\Theta$ is noted $\theta=(\gamma,\mu,\sigma)$. 

The restriction  $\widetilde L_n:\Theta\to [-\infty,+\infty)$ is continuous, so that for any compact $K\subset \Theta$, $\widetilde L_n$ is bounded and reaches its maximum on $K$. We can thus define $\widetilde \theta_n^K=(\widetilde \gamma_n^K,\widetilde \mu_n^K,\widetilde \sigma_n^K)$ such that
\begin{equation}\label{eq:tildethetaK}
\widetilde \theta_n^K=\argmax_{\theta\in K}\, \widetilde L_n(\theta).
\end{equation}
The following proposition is the key in the proof of Theorem~\ref{theo1}.
\begin{proposition}\label{prop:wald}  Let $\theta_0=(\gamma_0,0,1)$ and $K\subset\Theta$ be a compact neighborhood of $\theta_0$. Under the assumptions of Theorem \ref{theo1}, 
\[
\lim_{n\to+\infty}\widetilde \theta_n^K=\theta_0\quad a.s..
\]
\end{proposition}
The proof of Proposition \ref{prop:wald} relies on  an adaptation of  Wald's method for proving the consistency of  $M$-estimators (see Wald \cite{W49} or van der Vaart \cite{vdV98}  Theorem 5.14). The standard theory of $M$-estimation is designed for i.i.d. samples, while  we have to deal with the triangular array $\{(\widetilde M_{k,m})_{1\leq k\leq n}, n\geq 1\}$. We first state  two lemmas. 
\begin{lemma}\label{lem:wald1}
For all $\theta\in\Theta$, $G_{\gamma_0}[\ell_\theta]\leq G_{\gamma_0}[\ell_{\theta_0}]$ and the equality holds if and only if $\theta=\theta_0$.
\end{lemma}
\begin{proof}[Proof of Lemma \ref{lem:wald1}]
The quantity $G_{\gamma_0}[\ell_{\theta_0}-\ell_\theta]$ is the Kullback-Leibler divergence of th GEV distributions with parameters $\theta_0$ and $\theta$ and is known to be non-negative (see van der Vaart \cite{vdV98} section 5.5). It vanishes if and only if the two distributions agree. This occurs if and only if $\theta=\theta_0$ because the GEV model is identifiable.
\end{proof}
\begin{lemma}\label{lem:wald2}
For $B\subset\Theta$, define
\[
\ell_B(x)=\sup_{\theta\in B}\ell_\theta(x),\quad x\in\bbR.
\]
Let $\theta\in\Theta$ and  $B(\theta,\varepsilon)$ be the open ball in $\Theta$ with center $\theta$ and radius $\varepsilon>0$. Then,
\[
\lim_{\varepsilon\to 0} G_{\gamma_0}[\ell_{B(\theta,\varepsilon)}]= G_{\gamma_0}[\ell_\theta].
\]
\end{lemma}
\begin{proof}[Proof of Lemma \ref{lem:wald2}]
Proposition \ref{prop:GEV} implies 
\[
\ell_\theta(x)=\ell_\gamma((x-\mu)/\sigma)-\log\sigma\leq m_\gamma-\log\sigma.
\]
One deduce that if  $B$ is contained in  $(1,\bar\gamma]\times [\bar\sigma,+\infty)\times \bbR$ for some $\bar\gamma> -1$ and $\bar\sigma>0$ , then there exists $M(\bar\gamma,\bar\sigma)$ such that
\[
\ell_\theta(x)\leq M(\bar\gamma,\bar\sigma) \quad \mbox{for all}\ \theta\in B, x\in\bbR.
\]
Hence there exists $M>0$ such that function $M-\ell_{B(\theta,\varepsilon)}$ is non-negative for $\varepsilon$ small enough. The continuity of $\theta\mapsto \ell_\theta(x)$ on $\Theta$ implies 
\[
\lim_{\varepsilon\to 0}\ell_{B(\theta,\varepsilon)}(x)=\ell_\theta(x)\quad \mbox{for all}\ x\in\bbR. 
\]
Then, Fatou's Lemma entails
\[
G_{\gamma_0}\big[\liminf_{\varepsilon\to 0}(M-\ell_{B(\theta,\varepsilon)})\big]\leq \liminf_{\varepsilon\to 0} G_{\gamma_0}\big[M-\ell_{B(\theta,\varepsilon)}\big],
\]
whence we obtain
\[
\limsup_{\varepsilon\to 0} G_{\gamma_0}[\ell_{B(\theta,\varepsilon)}]\leq G_{\gamma_0}[\ell_\theta].
\]
On the other hand, $\theta\in B(\theta,\varepsilon)$ implies $G_{\gamma_0}[\ell_{B(\theta,\varepsilon)}]\geq G_{\gamma_0}[\ell_\theta]$. We deduce 
\[
\lim_{\varepsilon\to 0} G_{\gamma_0}[\ell_{B(\theta,\varepsilon)}]= G_{\gamma_0}[\ell_\theta].
\]
\end{proof}

\begin{proof}[Proof of Proposition \ref{prop:wald}]
In view of Lemmas \ref{lem:wald1} and \ref{lem:wald2}, for each $\theta\in K$ such that $\theta\neq \theta_0$, there exists  $\varepsilon_\theta>0$ such that 
\[
G_{\gamma_0}[\ell_{B(\theta,\varepsilon_\theta)}]< G_{\gamma_0}[\ell_{\theta_0}].
\]
Fix $\delta>0$. The set $\Delta=\{\theta\in K; \|\theta-\theta_0\|\geq \delta\}$ is compact and is covered by the open balls  $\{B(\theta,\varepsilon_\theta), \theta\in \Delta\}$. Let $B_i=B(\theta_i,\varepsilon_{\theta_i})$, $1\leq i\leq p$, be a finite subcover. Using the relation $\widetilde L_n(\theta)=\bbP_n[\ell_\theta]$, we see that
\[
\sup_{\theta\in\Delta} \widetilde L_n(\theta) \leq \max_{1\leq i\leq p} \bbP_n[\ell_{B_i}].
\]
The function $\ell_{B_i}$ is upper semi-continuous and bounded from above, so that Lemma \ref{lem:Fatou} entails
\[
\limsup_{n\to +\infty} \bbP_n[\ell_{B_i}]\leq G_{\gamma_0}[\theta_i]\quad \mbox{a.s.},
\]
whence
\begin{equation}\label{eq:wald1}
\limsup_{n\to +\infty} \sup_{\theta\in \Delta} \widetilde L_n(\theta)\leq \max_{1\leq i\leq p} G_{\gamma_0}[\theta_i] <G_{\gamma_0}[\theta_0]\quad \mbox{a.s.}.
\end{equation}
According to Lemma~\ref{lem:crucial}, $\bbP_n[\ell_{\theta_0}]\stackrel{a.s.}\longrightarrow G_{\gamma_0}[\ell_{\gamma_0}]$, so that
\begin{equation}\label{eq:wald2}
\liminf_{n\to +\infty} \sup_{\theta\in K} \widetilde L_n(\theta)\geq G_{\gamma_0}[\theta_0]\quad \mbox{a.s.}.
\end{equation}
Since $\widetilde\theta_n^K$ realizes the maximum of $\widetilde L_n$ over $K$, Equations \eqref{eq:wald1} and \eqref{eq:wald2} together entail that $\widetilde\theta_n^K\in K\setminus \Delta$ for large $n$. Equivalently, $\|\widetilde\theta_n^K-\theta_0\|<\delta$ for large $n$. Since $\delta$ is arbitrary, this proves the convergence $\widetilde\theta_n^K\stackrel{a.s.}\longrightarrow\theta_0$ as $n\to+\infty$.
\end{proof}

\begin{proof}[Proof of Theorem \ref{theo1}]
Let $K\subset \Theta$ be a compact neighborhood of $\theta_0$ as in Proposition \ref{prop:wald} and define $\widetilde \theta_n^K$ by Equation \eqref{eq:tildethetaK}. We prove that Theorem \ref{theo1}  holds true with the  sequence of estimators 
\[
(\widehat\gamma_n,\widehat\mu_n,\widehat\sigma_n)=(\widetilde\gamma_n^K,a_m\widetilde\mu_n^K+b_m,a_m\widetilde\sigma_n^K),\quad n\geq 1.
\]
According to Lemma \ref{lem:basic}, $(\widehat\gamma_n,\widehat\mu_n,\widehat\sigma_n)$ is a MLE if and only if $\widetilde L_n$ has a local maximum at $\widetilde \theta_n^K=(\widetilde\gamma_n^K,\widetilde\mu_n^K,\widetilde\sigma_n^K)$. Since $\widetilde \theta_n^K=\argmax_{\theta\in K} \widetilde L_n(\theta)$, this is the case as soon as $\widetilde L_n$ lies in the interior set $\mathrm{int}(K) $ of $K$. 
Proposition \ref{prop:wald} implies the almost surely convergence $\widetilde\theta_n^K\stackrel{a.s.}\longrightarrow \theta_0$ which is equivalent to Equation \eqref{eq:theo1.2}. Furthermore, since $\theta_0\in\mathrm{int}(K)$, this implies $\widetilde\theta_n^K\in \mathrm{int}(K)$ for large $n$ so that $(\widehat\gamma_n,\widehat\mu_n,\widehat\sigma_n)$ is a MLE for large $n$. This proves Equation \eqref{eq:theo1.1}.
\end{proof}

\section*{Acknowledgements}
C. Dombry is grateful to Laurens de Haan for suggesting the research problem in the first place and  for useful comments that improved greatly  an early  version of the manuscript.

\appendix
\section*{Appendix: Proof of Lemma~\ref{lem:crucial}}
We will use the following criterion.
\begin{lemma}\label{prop:A1}
Suppose $F\in D(G_{\gamma_0})$ and $\lim_{n\to+\infty} m(n)=+\infty$. We note $Y_{m}=\ell_{\gamma_0}(a_m^{-1}(M_{1,m}-b_m))$. If there exists a sequence  $(\alpha_n)_{n\geq 1}$ and $p>2$ such that
\[
\sum_{n\geq 1}n\bbP(|Y_{m}|> \alpha_n)<+\infty\quad \mbox{and}\quad 
\sup_{n\geq 1}\bbE[|Y_{m}|^p1_{\{|Y_m|\leq \alpha_n\}}]<+\infty,
\]
then Equation \eqref{eq:crucial} holds true.
\end{lemma}

\begin{proof}[Proof of lemma \ref{prop:A1}]
We note  $\mu=G_{\gamma_0}[\ell_{\gamma_0}]$ and we define
\[
Y_{k,m}= \ell_{\gamma_0}(a_m^{-1}(M_{k,m}-b_m))\quad \mbox{and}\quad S_n= \sum_{k=1}^n Y_{k,m}. 
\]
With these notations, \eqref{eq:crucial} is equivalent to $n^{-1}S_n \stackrel {a.s.}\to \mu$. 
We introduce the truncated variables 
\[
\widetilde Y_{k,m}=Y_{k,m}1_{\{|Y_{k,m}|\leq \alpha_n\}}\quad \mbox{and}\quad\widetilde S_n=\sum_{k=1}^n \widetilde Y_{k,m}. 
\]
Clearly,
\begin{eqnarray*}
\bbP[\widetilde S_n\neq  S_n]&\leq&  \bbP[\widetilde Y_{k,m}\neq Y_{k,m} \mbox{ for some }  k\in\{1,\ldots,n\} ]\nonumber\\
&\leq& n\bbP[|Y_{m}|> \alpha_n],
\end{eqnarray*}
so that  $\sum_{n\geq 1}n\bbP[|Y_{m}|> \alpha_n]<+\infty$ entails $\sum_{n\geq 1} \bbP[\widetilde S_n\neq  S_n]< +\infty$.
By the Borel-Cantelli Lemma, this implies that the sequences $(\widetilde S_n)_{n\geq 1}$ and $(S_n)_{n\geq 1}$ coincide eventually, 
whence $n^{-1}S_n \stackrel {a.s.}\to \mu$ if and only if $n^{-1}\widetilde S_n \stackrel {a.s.}\to \mu$. We now prove this last convergence. 

We first prove that $\bbE[\widetilde Y_{1,m}]\to \mu$. Indeed, by the continuous mapping theorem, the weak convergence \eqref{eq:convBM} implies $Y_{1,m}\Rightarrow \ell_{\gamma_0}(Z)$ with $Z\sim G_{\gamma_0}$. Since $\bbP[\widetilde Y_{1,m}\neq  Y_{1,m}]$ converges to $0$ as $n\to +\infty$, it also holds $\widetilde Y_{1,m}\Rightarrow \ell_{\gamma_0}(Z)$. Together with the condition $\sup_{n\geq 1}\bbE[|\widetilde Y_{1,m}|^p]<\infty$, this entails  $\bbE[\widetilde Y_{1,m}]\to \bbE[\ell_{\gamma_0}(Z)]=\mu$.  

Next, Theorem 2.10 in Petrov \cite{P95} provides the upper bound 
\[
\bbE[|\widetilde S_n-\bbE[\widetilde S_n]|^p]\leq C(p)n^{p/2}\bbE[|\widetilde Y_{1,m}-\bbE[\widetilde Y_{1,m}|^p]
\]
for some constant $C(p)>0$ depending only on $p$. Equivalently,
\[
\bbE[|n^{-1}\widetilde S_n-\mu_n|^p]\leq C(p)n^{-p/2}\bbE[|\widetilde Y_{1,m}-\mu_n|^p].
\]
with $\mu_n=\bbE[ \widetilde Y_{1,m}]$. Furthermore,  
\[
\bbE[|\widetilde Y_{1,m}-\mu_n|^p]\leq 2^{p-1}(\bbE[|\widetilde Y_{1,m}|^p]+|\mu_n|^p)
\]
is uniformly bounded by some constant $C>0$. By the Markov inequality, for all $\varepsilon>0$,  
\[
\bbP[|n^{-1}\widetilde S_n-\mu_n|\geq\varepsilon ]\leq \varepsilon^{-p}\bbE[|n^{-1}\widetilde S_n-\mu_n|^p]\leq \varepsilon^{-p}C(p)Cn^{-p/2}.
\]
Since $p>2$, it holds 
\[
\sum_{n\geq 1}\bbP[|n^{-1}\widetilde S_n-\mu_n|\geq\varepsilon ]<+\infty 
\]
and the Borel-Cantelli Lemma entails  $n^{-1}\widetilde S_n-\mu_n\stackrel{a.s.}\to 0$. 
Since $\mu_n\to\mu$, we deduce $n^{-1}\widetilde S_n\stackrel{a.s.}\to \mu$  which proves the Lemma. 
\end{proof}

\begin{proof}[Proof of Lemma \ref{lem:crucial}]
We prove that there exists a sequence $(\alpha_n)$ and $p>2$ satisfying   
\begin{equation}\label{eq:final1}
\sum_{n\geq 1} n\alpha_n^{-p}<+\infty
\end{equation}
and
\begin{equation}\label{eq:final2}
\sup_{n\geq 1}\bbE[(|Y_{m}|\wedge \alpha_n)^p]<+\infty.
\end{equation}
The Markov inequality yields
\[
\bbP[|Y_m|\geq \alpha_n] \leq \alpha_n^{-p}\bbE[(|Y_{m}|\wedge \alpha_n)^p]
\]
so that Equations \eqref{eq:final1} and \eqref{eq:final2} together entail 
\[
\sum_{n\geq 1}n\bbP(|Y_{m}|> \alpha_n)<+\infty\quad \mbox{and}\quad 
\sup_{n\geq 1}\bbE[|Y_{m}|^p1_{\{|Y_m|\leq \alpha_n\}}]<+\infty,
\]
This shows that Equations \eqref{eq:final1} and \eqref{eq:final2} together imply the assumptions of Lemma~\ref{prop:A1} and prove  Lemma \ref{lem:crucial}.

We first evaluate the quantity $\bbE[(|Y_{m}|\wedge \alpha_n)^p]$ from Equation \eqref{eq:final2}. Recall that $Y_m=\ell_{\gamma_0}((M_{1,m}-b_m)/a_m)$. 
It is well known that the random variable $X_i$ with distribution function $F$ has the same distribution as the random variable $F^\leftarrow(V)$, with $V$ a uniform random variable on $(0,1)$. We deduce that the random variable $M_{1,m}=\vee_{i=1}^m X_i$ has the same distribution as $F^\leftarrow(V_m)$, with $V_m$ a random variable with distribution $mv^{m-1}1_{(0,1)}(v)dv$ (this is the distribution of the maximum of $m$ i.i.d. uniform random variables on $[0,1]$). Hence,
\[
\bbE[(|Y_{m}|\wedge \alpha_n)^p]=\int_0^1 (|\ell_{\gamma_0}((F^\leftarrow(v)-b_m)/a_m)|\wedge \alpha_n)^p mv^{m-1}dv.
\] 
The relations $U(x)=F^\leftarrow(1-1/x)$ and $b_m=U(m)$ together with the change of variable $v=1-1/(mx)$ yield
\[
\bbE[(|Y_{m}|\wedge \alpha_n)^p]=\int_{1/m}^\infty \big(|\ell_{\gamma_0}(\widetilde U_m(x))|\wedge \alpha_n\big)^p \big(1-\frac{1}{mx}\big)^{m-1}x^{-2}dx
\]
where
\[
\widetilde U_m(x)=\frac{U(mx)-U(m)}{a_m}.
\]
We now provide an upper bound for the integral and we use the following estimates. There exists a constant $c>0$ such that 
\[
 |\ell_{\gamma_0}(y)|\leq \left\{ \begin{array}{ll} c(1+{\gamma_0}y)^{-1/\gamma_0}, &y< 0,\\ (1+1/\gamma_0)\log(1+\gamma_0y)+1, & y\geq 0\end{array}\right..
\]
Note that $\widetilde U_m(x)$ is positive for $x>1$ and negative for $x<1$. Furthermore, for all $x\geq 1/m$ and $m\geq 2$,
\[
\big(1-\frac{1}{mx}\big)^{m-1}\leq \exp(-(m-1)/(mx))\leq \exp(-1/(2x)).
\]
Using these estimates, we obtain the following upper bound: for $m\geq m_0$ ($m_0$ to be precised later),
\begin{equation}\label{eq:sum}
\bbE[(|Y_{m}|\wedge \alpha_n)^p]\leq I_1+I_2+I_3
\end{equation}
with
\begin{eqnarray*}
I_1&=&\int_{1/m}^{m_0/m} \alpha_n^p \exp(-1/(2x))x^{-2}dx,\\
I_2&=&\int_{m_0/m}^1 c^p\big(1+{\gamma_0}\widetilde U_m(x)\big)^{-p/\gamma_0}\exp(-1/(2x))x^{-2}dx,\\
I_3&=&\int_1^\infty \big((1+1/\gamma_0)\log\big(1+\gamma_0\widetilde U_m(x)\big)+1 \big)^p\exp(-1/(2x))x^{-2}dx.
\end{eqnarray*}
The integral $I_1$ can be computed explicitely and 
\begin{equation}\label{eq:I1}
I_1 \leq 4\alpha_n^p\exp(-m/(2m_0)).
\end{equation}
To estimate $I_2$ and $I_3$, we need upper and lower bounds for $\widetilde U_m(x)$ and we have to distinguish between the three cases $\gamma_0>0$,  $\gamma_0\in (-1,0)$ and $\gamma=0$.\\
{\bf Case $\gamma_0>0$:} According to Theorem \ref{theo:DA}, the function $U$ is regularly varying at infinity with index $\gamma_0>0$ and
\[
1+\gamma_0\widetilde U_m(x)=1+\gamma_0\frac{U(mx)-U(m)}{a_m}=\frac{U(mx)}{U(m)}.
\]
We use then  Potter's bound (see e.g. Proposition B.1.9 in \cite{dHF06}): for all $\varepsilon >0$, there exists $m_0\geq 1$ such that for $m\geq m_0$ and $mx\geq m_0$
\[
(1-\varepsilon)x^{\gamma_0}\min(x^\varepsilon,x^{-\varepsilon})\leq \frac{U(mx)}{U(m)}\leq (1+\varepsilon)x^{\gamma_0}\max(x^\varepsilon,x^{-\varepsilon}).
\]
We fix $\varepsilon \in(0,\gamma_0)$ and choose $m_0$ accordingly. Using the lower Potter's bound to estimate $I_2$ and the upper Potter's bound to estimate $I_3$, we get
\begin{eqnarray*}
I_2&\leq&\int_{m_0/m}^1 c^p\big((1-\varepsilon)x^{\gamma_0+\varepsilon}\big)^{-p/\gamma_0}\exp(-1/(2x))x^{-2}dx\\
&\leq &c^p(1-\varepsilon)^{-p/\gamma_0} \int_{0}^1 x^{-2-p-p\varepsilon/\gamma_0}\exp(-1/(2x))dx,
\end{eqnarray*}
and
\[
I_3\leq\int_1^\infty \big((1+1/\gamma_0)\log\big((1+\varepsilon)x^{\gamma_0+\varepsilon}  \big)+1 \big)^p\exp(-1/(2x))x^{-2}dx.
\]
These integrals are finite and this implies that $I_2$ and $I_3$ are uniformly bounded for $m\geq m_0$. From Equations \eqref{eq:sum} and \eqref{eq:I1}, we obtain
\[
\bbE[(|Y_{m}|\wedge \alpha_n)^p]\leq 4\alpha_n^p\exp(-m/(2m_0))+C,
\]
for some constant $C>0$. Finally, we set $\alpha_n^p\exp(-m/(2m_0))=1$, i.e. $\alpha_n=\exp(m/(p2m_0))$. Equation \eqref{eq:final2} is clearly satisfied and 
\[
n\alpha_n^{-p}=\exp[\log n -m/(2m_0)]=\exp[-(m/(2m_0\log n)-1)\log n].
\]
We check easily that the condition $\lim_{n\to +\infty}\frac{m(n)}{\log n}=+\infty$ implies  Equation \eqref{eq:final1}.\\
{\bf Case $\gamma_0<0$:} It follows from Theorem \ref{theo:DA} that the function $t\mapsto U(\infty)-U(t)$ is regularly varying at infinity with index $\gamma_0<0$ and that
\[
1+\gamma_0\widetilde U_m(x)=1+\gamma_0\frac{U(mx)-U(m)}{a_m}=\frac{U(\infty)-U(mx)}{U(\infty)-U(m)}.
\]
Then, the Potter's bounds  become: for all $\varepsilon >0$, there exists $m_0\geq 1$ such that for $m\geq m_0$ and $mx\geq m_0$
\[
(1-\varepsilon)x^{\gamma_0}\min(x^\varepsilon,x^{-\varepsilon})\leq \frac{U(\infty)-U(mx)}{U(\infty)-U(m)}\leq (1+\varepsilon)x^{\gamma_0}\max(x^\varepsilon,x^{-\varepsilon}).
\]
Using this, the proof is completed in the same way as in the case $\gamma_0>0$  with straightforward modifications.\\
{\bf Case $\gamma_0=0$:} In this case, Theorem B.2.18 in \cite{dHF06} implies that for all $\varepsilon>0$, there exists $m_0\geq 1$ such that for $m\geq m_0$ and $mx\geq m_0$,
\[
\Big|\frac{U(mx)-U(m)}{a_m}-\log x\Big|\leq \varepsilon \max(x^\varepsilon,x^{-\varepsilon}).
\]
Equivalently, for $m\geq m_0$ and $mx\geq m_0$,
\[
\log x - \varepsilon \max(x^\varepsilon,x^{-\varepsilon})\leq \widetilde U_m(x) \leq \log x + \varepsilon \max(x^\varepsilon,x^{-\varepsilon}).
\]
Using the lower bound to estimate $I_2$ and the upper bound to estimate $I_3$, we obtain
\begin{eqnarray*}
I_2&=&\int_{m_0/m}^1 c^p\exp(-p\widetilde U_m(x))\exp(-1/(2x))x^{-2}dx\\
&\leq&c^p \int_{0}^1\exp(-p\log x +p \varepsilon x^{-\varepsilon}-1/(2x))x^{-2}dx,
\end{eqnarray*}
and
\begin{eqnarray*}
I_3&=&\int_1^\infty \big(\widetilde U_m(x)+1 \big)^p\exp(-1/(2x))x^{-2}dx\\
&\leq&\int_1^\infty \big(\log x + \varepsilon x^{\varepsilon}+1 \big)^p\exp(-1/(2x))x^{-2}dx.
\end{eqnarray*}
For $\varepsilon\in(0,1/p)$, the integrals appearing in the upper bounds are finite and independent of $m$. This shows that $I_2$ and $I_3$ are uniformly bounded for $m\geq m_0$. The proof is then completed as in the case $\gamma_0>0$.
\end{proof}

\bibliographystyle{plain}
\bibliography{BiblioMLE}
\end{document}